\documentclass{amsart}

\usepackage{amsmath}
\usepackage{amssymb}
\usepackage{commath}
\usepackage{mathtools}
\usepackage[all]{xy}
\usepackage{graphicx} 

\usepackage{bigints}

\usepackage{hyperref}

\allowdisplaybreaks

\theoremstyle{plain}
\newtheorem{theorem}{Theorem}
\newtheorem{lemma}[theorem]{Lemma}
\newtheorem{proposition}[theorem]{Proposition}
\newtheorem{corollary}[theorem]{Corollary}
\theoremstyle{definition}
\newtheorem{definition}[theorem]{Definition}
\newtheorem{remark}[theorem]{Remark}

\numberwithin{theorem}{subsection}
\numberwithin{equation}{section}

\newcommand{\B}{\mathbb{B}}

\newcommand{\C}{\mathbb{C}}
\newcommand{\HH}{\mathbb{H}}
\newcommand{\Sbb}{\mathbb{S}}
\newcommand{\Mbb}{\mathbb{M}}
\newcommand{\R}{\mathbb{R}}

\newcommand{\Z}{\mathbb{Z}}

\newcommand{\Spe}{\mathrm{Sp}}
\newcommand{\SO}{\mathrm{SO}}
\newcommand{\SU}{\mathrm{SU}}
\newcommand{\OO}{\mathrm{O}}
\newcommand{\GL}{\mathrm{GL}}

\newcommand{\spe}{\mathfrak{sp}}
\newcommand{\frakg}{\mathfrak{g}}
\newcommand{\frakk}{\mathfrak{k}}

\newcommand{\frakm}{\mathfrak{m}}
\newcommand{\frakl}{\mathfrak{l}}

\newcommand{\sgn}{\mathrm{sgn}}

\newcommand{\im}{\operatorname{Im}}
\newcommand{\re}{\operatorname{Re}}

\newcommand{\diag}{\operatorname{diag}}
\newcommand{\Ad}{\operatorname{Ad}}

\newcommand{\cF}{\mathcal{F}}
\newcommand{\cM}{\mathcal{M}}

\begin{document}

\title[Lie theory of slice geometry]{Lie theory of the slice Riemannian geometry on the quaternionic unit ball}

\author{Raul Quiroga-Barranco}
\address{Centro de Investigaci\'on en Matem\'aticas, Guanajuato, M\'exico}
\email{quiroga@cimat.mx}



\maketitle


\begin{abstract}
	The quaternionic unit ball carries a Riemannian metric built using regular M\"obius transformations: the slice Riemannian metric. We prove that the geometry induced by this metric is strongly related to the group $\mathrm{Sp}(1,1)$. We also develop the foundations for a Lie theoretic study of the slice Riemannian metric. In particular, we compute its isometry group and prove that it is built from symmetries of the Lie group $\mathrm{Sp}(1,1)$. We also compare the slice Riemannian geometry with the quaternionic Poincar\'e geometry, where the latter is considered within the setup of Riemannian symmetric spaces.
\end{abstract}

\section{Introduction}
\label{sec:intro}
An outstanding feature of hypercomplex analysis is its ability to bring out interesting properties in seemingly already understood mathematical objects. This has been particularly so with the function theory associated to the notion of slice regularity, also called Cullen regularity. The research in this area is quite large and varied. We refer to \cite{ColomboSabadiniStruppaFunctionalBook,GentiliStoppatoStruppa2ndEd} 
for the fundamentals of slice regular functions of one variable.
In this work we will focus in the interplay between slice regular functions and Riemannian geometry on the quaternionic unit ball $\B$. 

From the classical point of view, $\B$ carries a Riemannian structure whose isometries are (essential all) given by classical M\"obius transformations: this is the so called quaternionic hyperbolic metric on $\B$ also known as the quaternionic Poincar\'e metric. This leads to consider the quaternionic symplectic Lie group $\Spe(1,1)$. The fundamental properties of the corresponding geometry are very well known (see~\cite{BisiGentiliMobius,BisiStoppatoMobius}) and is a subset of the theory of Riemannian symmetric spaces (see~\cite{Helgason}).

On the other hand, as consequence of the development of slice regular function theory, the notion of regular M\"obius transformations on $\B$ has been introduced (see~\cite{GentiliStoppatoStruppa2ndEd}). One would expect this to yield new Riemannian structures on $\B$. This was proved to be the case in \cite{QB-SliceKahler}, where the notion of slice Riemannian metric on $\B$ was defined. More precisely, there is a metric on $\B$ that can be built from regular M\"obius transformations and that behaves nicely with respect to them. Furthermore, in \cite{QB-SliceKahler} it was proved that the slice Riemannian metric coincides with the Riemannian metric built in \cite{ArcozziSarfatti} from a suitable quaternionic Hardy space defined in $\B$.

The (classical) quaternionic Poincar\'e metric on $\B$ may be studied through the group $\Spe(1,1)$. This is a particular case of the study of Riemannian symmetric spaces through Lie theoretic techniques. Hence, it is natural to ask whether a similar study may be performed for the slice Riemannian metric in $\B$. The main contribution of this paper is to lay the foundations of such study. More precisely, we prove that the Lie theoretic properties of $\Spe(1,1)$ are tightly related to the geometric properties of the slice Riemannian metric through the use of slice regular function theory.

To achieve our goal, we recall in Section~\ref{sec:RiemSymmB} the basic properties, geometric and Lie theoretic, of the quaternionic Poincar\'e geometry on $\B$. We start by reviewing some general notions of Riemannian symmetric spaces in subsection~\ref{subsec:RiemSymmSpaces} which we then specialize to $\B$ in subsection~\ref{subsec:RiemSymmB}.

The geometry and Lie theory of the slice Riemannian metric of $\B$ is studied in Section~\ref{sec:SliceRiemB}, where the main results can be found. Subsection~\ref{subsec:sliceregular} recalls the basics of regular function theory, including the notion of regular M\"obius transformation. The definition of the slice Riemannian metric on $\B$ is given in subsection~\ref{subsec:sliceRiem}, where its isometry group is computed in Theorem~\ref{thm:Iso(B,g)}. The Lie theory associated to the slice Riemannian metric on $\B$ is developed in subsection~\ref{subsec:sliceRiemLie}. One of our main tools is Theorem~\ref{thm:SliceRegularBasQuotient}, which is originally found in \cite{QB-cM(B)-asmanifold} and that realizes $\B$ as a double coset quotient of the group $\Spe(1,1)$; such realization makes use of regular M\"obius transformations. We prove in Theorem~\ref{thm:Iso(B,g)fromSp(1,1)} that an index $2$ subgroup (the orientation preserving subgroup) of the isometry group for the slice Riemannian metric can be described in terms of left and right translations in $\Spe(1,1)$. In other words, the symmetries of the slice Riemannian metric are given by the symmetries of $\Spe(1,1)$. We construct in Theorem~\ref{thm:G=K1BK2-fromsliceregular} a decomposition of $\Spe(1,1)$, up to diffeomorphism, that parallels the one classically obtained for Riemannian symmetric spaces. Such decomposition of $\Spe(1,1)$ is interesting by itself and may be considered as a useful application of hypercomplex analysis to Lie theory.

Finally, we present in subsection~\ref{subsec:SymmVSSlice} a detailed comparison of the Lie theory associated to the quaternionic Poincar\'e geometry and the slice Riemannian geometry. The properties obtained here are consequence of the theory developed in previous subsections. We observe some interesting similarities between both cases that might be considered unexpected. At the same time, some strong differences are noted. For example, the isometry group of the quaternionic Poincar\'e metric acts transitively, while the orbits of the isometry group for the slice Riemannian metric are (real) $3$-dimensional with the exception of a $1$-dimensional orbit: $(-1,1) = \R\cap\B$.

\section{Riemannian symmetric space structure of $\B$}
\label{sec:RiemSymmB}
\subsection{Riemannian symmetric spaces}\label{subsec:RiemSymmSpaces}
We recall some of the basic notions related to Riemannian symmetric spaces. Our main references will be \cite{Helgason}, with particular emphasis in Chapter~IV. Nevertheless, we will provide references below when more details are called for.

\begin{definition}\label{def:RiemSymmSpace}
	A Riemannian symmetric space is a connected Riemannian manifold $M$ so that, for every $p \in M$, there exists an isometry of $M$ that has $p$ as unique fixed point in a neighborhood of $p$.
\end{definition}

We have the following characterization of Riemannian symmetric spaces which highlights two of their most important features: homogeneity and completeness. 

\begin{proposition}\label{prop:symmhomogeneous}
	Let $M$ be a connected Riemannian manifold. Then, $M$ is a Riemannian symmetric space if and only if it is homogeneous, as a Riemannian manifold, and for some $p_0 \in M$ there is an isometry of $M$ that has $p_0$ as an isolated fixed point. Furthermore, in this case $M$ is a complete Riemannian manifold, i.e.~every geodesic is defined in the whole real line $\R$.
\end{proposition}
\begin{proof}
	To prove sufficiency note that, if $p_0 \in M$ and $f_0$ is an isometry that fixes $p_0$ as in the statement, then homogeneity implies that for every $p \in M$ there exists an isometry $f$ of $M$ mapping $f(p_0) = p$. It is now easy to prove that $f \circ f_0 \circ f^{-1}$ is an isometry that has $p$ as an isolated fixed point. The necessity follows from the discussion in \cite[Chapter~IV,~Section~3]{Helgason}. The last claim is established in the proof of \cite[Chapter~IV,~Lemma~3.1]{Helgason}.
\end{proof}

The isometries considered in the definition of Riemannian symmetric space can be geometrically described in a way that allows us to see their uniqueness. This is the content of the next result, which is a consequence of \cite[Chapter~IV,~Lemma~3.1]{Helgason}.

\begin{proposition}\label{prop:involutivesymmetry}
	Let $M$ be a Riemannian symmetric space. If $f$ is an isometry of $M$ with $p$ as isolated fixed point, then it reverses the parametrization of geodesics through $p$. More precisely, if $\gamma : \R \rightarrow M$ is a geodesic in $M$ with $\gamma(0) = p$, then $f(\gamma(t)) = \gamma(-t)$ for all $t \in \R$. In particular, for every $p \in M$ there is a unique isometry $f$ with $p$ as isolated fixed point and it is involutive: $f^2 = id_M$.
\end{proposition}

It is well known that the isometry group of an arbitrary Riemannian manifold has a natural Lie group structure, compatible with the compact-open topology, such that its action on the Riemannian manifold is smooth (see \cite[Chapter~II]{KobayashiTransformations}). For a Riemannian symmetric space this Lie group of isometries has further properties, many of them very much related to the geometry of the space. We now proceed to describe those that will be relevant to this work. First, we recall some basic Lie theory notions.

\begin{remark}\label{rmk:exp}
	It will be useful to consider the exponential map for Lie groups. We describe here the definition and most basic properties, while referring to \cite{WarnerBook} for further details. Let $G$ be a Lie group with Lie algebra $\frakg$. We observe that $\frakg$ is naturally identified with the tangent space of $G$ at the identity element. A one-parameter subgroup of $G$ is a smooth map $\gamma : \R \rightarrow G$ that satisfies $\gamma(s + t) = \gamma(s) \gamma(t)$, for every $s,t \in \R$; in other words, $\gamma$ is a smooth homomorphism of Lie groups. For every $X \in \frakg$, there exists a unique one-parameter subgroup $\gamma_X : \R \rightarrow G$ such that $\gamma_X'(0) = X$. The exponential map of $G$ is the map $\exp : \frakg \rightarrow G$ given by $\exp(X) = \gamma_X(1)$, for every $X \in \frakg$. It follows that $\exp$ is smooth and satisfies $\exp(tX) = \gamma_X(t)$, for every $X \in \frakg$ and $t \in \R$. Furthermore, this last identity characterizes the exponential map.
\end{remark}

Most of the Lie theory properties of a Riemannian symmetric space are derived from the facts stated in the next result. We refer to \cite[Chapter~IV,~Theorem~3.3]{Helgason} for the proof and further details. Note that, for convenience, we will sometimes consider isometric Lie group actions on the right. 

\begin{proposition}\label{prop:RiemSymmProperties}
	Let $M$ be a Riemannian symmetric space and $G = I_0(M)$ be the identity component of the Lie group of isometries of $M$. For a given point $p_0 \in M$, let us denote by $K$ the isotropy in $G$ of $p_0$ (the subgroup of $G$ consisting of elements that fix $p_0$). Then, the following properties hold.
	\begin{enumerate}
		\item The group $K$ is compact and the map $K\backslash G \rightarrow M$ given by $Kg \mapsto p_0\cdot g$ is a $G$-equivariant diffeomorphism.
		
		\item Let $f_0$ be an isometry of $M$ that has $p_0$ as isolated fixed point. Then, the map $\sigma : G \rightarrow G$ given by $g \mapsto f_0 g f_0$ is an involutive isomorphism of Lie groups. Furthermore, $K$ lies between the fixed point subgroup $K_\sigma$ of $\sigma$ and the identity component of $K_\sigma$.
		
		\item Let us denote by $\frakg$ and $\frakk$ the Lie algebras of $G$ and $K$, respectively. For $\dif \sigma_I : \frakg \rightarrow \frakg$ the differential of $\sigma$ at the identity element $I \in G$ we have $\frakk = \{ X \in \frakg : \dif \sigma_I(X) = X \}$. If we consider $\frakm = \{ X \in \frakg : \dif \sigma_I(X) = -X \}$, then we also have $\frakg =  \frakk \oplus \frakm$.
		
		\item With the previous notation, if $\pi : G \mapsto M$ denotes the orbit map $g \mapsto p_0 \cdot g$, then the differential $\dif \pi_I : \frakg \rightarrow T_{p_0} M$ satisfies the following properties.
		\begin{enumerate}
			\item The map $\dif \pi_I$ has kernel $\frakk$ and maps $\frakm$ isomorphically onto $T_{p_0} M$.
			\item The geodesic through $p_0$ with initial velocity $\dif \pi_I(X)$, where $X \in \frakm$, is given by
			\[
				t \mapsto p_0\cdot \exp(tX),
			\]
		\end{enumerate}
		In particular, this describes all geodesics of $M$ through $p_0$.
	\end{enumerate}
\end{proposition}

\begin{remark}\label{rmk:transvections}
	With the notation of Proposition~\ref{prop:RiemSymmProperties}, for every $X \in \frakm$, the isometries defined by $\exp(tX)$ ($t \in \R$) are called transvections. Hence, we can rephrase Proposition~\ref{prop:RiemSymmProperties}(4) saying that the geodesics through $p_0$ are orbits of transvections considered as one-parameter subgroups. We refer to \cite[p.209]{Helgason} for further details.
\end{remark}

From now on, we will consider Riemannian symmetric spaces of non-compact type, i.e.~those for which the identity component of the isometry group is non-compact semisimple without (local) compact factors. We refer to \cite[Chapter~VI]{Helgason} for further details and definitions.
In this case, a Riemannian symmetric space induces a decomposition of the identity component of its isometry group in terms of the space and a compact group. This fact is stated in the next result, which is consequence of \cite[Chapter~VI,~Theorem~1.1]{Helgason}.

\begin{proposition}\label{prop:G=KM}
	With the notation of Proposition~\ref{prop:RiemSymmProperties}, if $M$ is a Riemannian symmetric space of non-compact type, then the map $K \times \frakm \rightarrow G$, given by 
	\[
		(k,X) \mapsto k\exp(X),
	\]
	is a diffeomorphism.
\end{proposition}

\begin{remark}\label{rmk:G=KM}
	We note that the version of the previous result stated in \cite{Helgason} makes use of isometric left actions. Since we will consider right actions of the isometry group, the product that defines the map above is reversed with respect to that reference. On the other hand, this result holds with more generality. More precisely, it is possible to consider Riemannian symmetric pairs $(G,K)$ where $G$ is a cover of the identity component of the isometry group (see~\cite[Chapter~VI]{Helgason}). We will have opportunity to use this fact.
\end{remark}

The homogeneous realization $K \backslash G$ of a Riemannian symmetric space $M$ of non-compact type given by Proposition~\ref{prop:RiemSymmProperties} encompasses all the symmetries of $M$ through those of $G$. This can be explained as follows. For the Lie group structure of $G$ there are three natural types of symmetries: for $g \in G$ the left translations $L_g$, the right translations $R_g$ and the conjugations $C_g$, which are defined as maps $G \rightarrow G$ given by the following expressions 
\begin{equation}\label{eq:LRCmaps}
	L_g(h) = gh, \quad R_g(h) = hg, \quad C_g(h) = L_g \circ R_{g^{-1}}(h) = ghg^{-1}.
\end{equation}
Furthermore, the conjugations yield the identity component of the automorphism group of $G$ when this group is semisimple. This follows from the fact that every automorphism of a group induces one of the corresponding Lie algebra (see~\cite[Chapter~II,~Lemma~1.12]{Helgason}) and that every automorphism of a semisimple Lie algebra comes from a conjugation (see~\cite[Chapter~II,~Corollary~6.5]{Helgason}). The natural question to ask here is which of these symmetries of $G$ descend to corresponding ones of $M$. The answer is provided in the next result, whose proof is based on the Lie theory of symmetric spaces. 

\begin{proposition}\label{prop:RiemSymmCentralizers}
	Let $M$ be a Riemannian symmetric space of non-compact type with homogeneous realization $K\backslash G$ given by Proposition~\ref{prop:RiemSymmProperties}. In particular, $G$ is a connected non-compact semisimple Lie group without compact factors. Let us assume that $G$ has finite center. Then, a diffeomorphism $\widetilde{f} : G \rightarrow G$ from the list \eqref{eq:LRCmaps} induces a diffeomorphism $f : M \rightarrow M$ so that the diagram
	\[
	\xymatrix{
		G\ar[rr]^{\widetilde{f}} \ar[d] && G\ar[d] \\
		M\ar[rr]^f && M
	}
	\]
	commutes, where the vertical arrows are the natural projections $G \rightarrow K \backslash G \simeq M$, exactly in the following cases
	\begin{enumerate}
		\item $\widetilde{f} = R_g$, for some $g \in G$, which induces the map: $f(Kx) = K(xg)$, where $Kx \in K\backslash G$.
		\item $\widetilde{f} = C_g$, for some $g \in K$, which induces the map: $f(Kx) = K(gxg^{-1}) = K(xg^{-1})$, where $Kx \in K\backslash G$.
	\end{enumerate}
	Hence, the maps from (2) form a subgroup of those from (1). Furthermore, all such maps $f$, thus obtained, yield the identity component of the isometry group of $M$. In particular, a diffeomorphism $f$ of $M$ belongs to the identity component of its isometry group if and only if it is induced by a map $\widetilde{f}$ which is either a right or left translation of $G$ that commutes with the left action of $K$.
\end{proposition}
\begin{proof}
	The right translations $R_g$ clearly induce corresponding diffeomorphisms on $M \simeq K \backslash G$, and the conjugations are expressed as $C_g = L_g \circ R_{g^{-1}}$. Hence, to determine the maps from the list \eqref{eq:LRCmaps} that descend to $M$ according to the diagram in the statement, it is enough to consider left translations.
	
	A left translation $L_g$, given by some $g \in G$, descends to a well defined diffeomorphism of  $K \backslash G$ if and only if $ghx \in Kgx$ for every $h \in K$ and $x \in G$. This is equivalent to the condition on $g \in G$ given by $C_g(K) = K$.	By considering the differential of $C_g$, which defines the adjoint representation $\Ad$ of $G$ (see~\cite[Chapter~II,~Section~5]{Helgason}), we conclude that this is equivalent to
	\[
		\Ad(g)(\frakk) = \frakk,
	\]
	where $\frakk$ is the Lie algebra of $K$. This is a consequence of the properties of the adjoint representation and the fact that $K$ is necessarily connected (see \cite[Chapter~VI,~Theorem~1.1]{Helgason}). Using again the properties of the adjoint representation, the subgroup of elements $g \in G$ that satisfy the previous identities has Lie algebra
	\[
		\frakl = \{ X \in \frakg : [X,\frakk] \subset \frakk \}.
	\]
	By considering the decomposition $\frakg = \frakk \oplus \frakm$ introduced in Proposition~\ref{prop:RiemSymmProperties}(3) and using that it defines an orthogonal symmetric pair of non-compact type (see \cite[Chapter~V,~Section~1]{Helgason}) we conclude that
	$\frakl = \frakk$. Hence, the identity component of the subgroup of $g \in G$ such that $L_g$ descends to $M$ is precisely $K$. Then, it follows from Proposition~\ref{prop:G=KM} that $K$ is precisely the whole subgroup of such elements. In this case, i.e.~for $g \in K$, the map $L_g$ induces the identity map on $K\backslash G \simeq M$, and so the conjugation $C_g$ and the right translation $R_{g^{-1}}$ induce the same map on $M$. Since the right $G$-action commutes with left $K$-action, the last claim is now obvious.
	
	The previous arguments complete the description stated in (1) and (2). That the latter is a subgroup of the former follows from an easy computation. Hence, we conclude that all symmetries of $G$ from \eqref{eq:LRCmaps}, that descend to $M$, induce the right $G$-action on $M$ which is precisely the identity component of the isometry group of $M$ (see~Proposition~\ref{prop:RiemSymmProperties}).
\end{proof}

\begin{remark}\label{rmk:RiemSymmCentralizers}
	Proposition~\ref{prop:RiemSymmCentralizers} shows that the connected isometry group $I_0(M)$ of a Riemannian symmetric space $M = K \backslash G$ of non-compact type is built out of the symmetries of $G$ coming from the right $G$-action on itself. As stated in this proposition, both a left or a right translation of $G$ need to commute with the left $K$-action in order to yield a symmetry of $M$. Considering the action on the right this leads us to the whole group $G$. However, on the left this leads to consider the subgroup
	\[
		C_G(K) = \{ g \in G : gk = kg \text{ for all } k \in K \},
	\]
	which is known as the centralizer of $K$ in $G$. Hence, the group $I_0(M)$ can be alternatively described as the group of diffeomorphisms coming from either of the following symmetries of $G$
	\begin{enumerate}
		\item $\{ L_g : g \in C_G(K) \}$,
		\item $\{ R_g : g \in G = C_G(\{I\}) \}$,
	\end{enumerate}
	where we have used the trivial fact $G = C_G(\{I\})$. We conclude that $I_0(M)$ comes from the centralizers of $K$ and of $\{I\}$. It is important to note the very much opposite nature of $C_G(K)$ and $C_G(\{I\})$. The latter is all of $G$, while the former can be seen (using arguments from the proof of Proposition~\ref{prop:RiemSymmCentralizers}) to be the center of $G$, which is assumed to be finite.
\end{remark}

\subsection{The quaternionic unit ball $\B$ as Riemannian symmetric space}
\label{subsec:RiemSymmB}
There is a well known and natural Riemannian symmetric space structure on the quaternionic unit ball
\[
	\B = \{ q \in \HH : |q| < 1 \},
\]
where $\HH$ denotes the division ring of quaternions. Here, we will describe some of the basic features of such structure of $\B$.

Let us consider the quaternionic symplectic Lie group with signature $(1,1)$, which we will define as
\[
	\Spe(1,1) = \{ A \in M_2(\HH) : A^* I_{1,1} A = I_{1,1} \},
\]
where $M_2(\HH)$ is the algebra of $2 \times 2$ quaternionic matrices and $A^*$ denotes the conjugate transpose of a matrix $A$. From now on, we will denote by $\diag(p,q)$ the diagonal $2 \times 2$ matrix whose diagonal is given by the ordered pair $(p,q)$. As a particular case, we will denote $I_{1,1} = \diag(1,-1)$.

Recall that there is a smooth right action of $\Spe(1,1)$ on $\B$ defined by
\begin{equation}\label{eq:MobiusClassical}
	q \cdot A = (q c + d)^{-1} (qa + b) = F_A(q), 
\end{equation}
where the matrix $A \in \Spe(1,1)$ has entries given by
\[
	A =
	\begin{pmatrix}
		a & c \\
		b & d
	\end{pmatrix}.
\]
The map $F_A$ thus defined will be called a \textbf{classical M\"obius transformation}. In particular, the collection $\Mbb(\B)$ of all such transformations is a group so that the map $A \mapsto F_A$ is an anti-homomorphism. 

It is well-known (see \cite{ArcozziSarfatti,BisiGentiliMobius,BisiStoppatoMobius,QB-SliceKahler}) that the $\Spe(1,1)$-action \eqref{eq:MobiusClassical} is isometric for the quaternionic Poincar\'e metric on $\B$ given by
\begin{equation}\label{eq:PoncaireMetric}
	(\alpha,\beta) \mapsto \widehat{g}_q(\alpha,\beta) = 
			\frac{\re(\alpha \overline{\beta})}{(1-|q|^2)^2},
\end{equation}
for every $q \in \B$ and $\alpha,\beta \in T_q\B$. From now, we will use the pair $(\B, \widehat{g})$ to denote the corresponding Riemannian manifold, i.e.~the quaternionic hyperbolic geometry on $\B$ defined by $\widehat{g}$. This notation will avoid any confusion when introducing other Riemannian metrics, which is actually one of the main goals of this work: to consider alternative geometries on $\B$ besides the quaternionic hyperbolic one.

The $\Spe(1,1)$-action on $\B$ given by \eqref{eq:MobiusClassical} is transitive with the subgroup $\{\pm I_2\}$ providing the only elements acting trivially. Since the map $q \mapsto -q$ is clearly an isometry for the metric \eqref{eq:PoncaireMetric} with $0$ as unique fixed point, it follows from Proposition~\ref{prop:symmhomogeneous} that $(\B,\widehat{g})$ is a Riemannian symmetric space. The isotropy subgroup of $\Spe(1,1)$ for $0$ is precisely
\[
	\Spe(1) \times \Spe(1) = \{ \diag(u,v) : u,v \in \Spe(1) \},
\]
as one can easily compute. It follows that $\B$ admits the homogeneous space realization
\begin{equation}\label{eq:BasSymmQuotient}
	\B \simeq (\Spe(1) \times \Spe(1)) \backslash \Spe(1,1).
\end{equation}

\begin{remark}\label{rmk:Sp(1,1)vsHelgason}
	The concrete Lie groups considered in \cite{Helgason} are all given by complex matrices. Hence, to apply the Lie theory developed in that reference, we need to discuss the Lie group $\Spe(1,1)$ and its realizations. 
	
	Let us consider the map $\psi : M_2(\HH) \rightarrow M_4(\C)$ defined by
	\begin{equation}\label{eq:psi}
		\psi(Z + W j) =
			\begin{pmatrix*}[r]
				Z & W \\
				-\overline{W} & \overline{Z}
			\end{pmatrix*},
	\end{equation}
	where $Z, W \in M_2(\C)$. It is a straightforward exercise to prove that $\psi$ is a monomorphism (injective homomorphism) of real algebras that maps $I_2$ to $I_4$. Hence, the restriction of $\psi$ to $\GL(2,\HH)$ defines a Lie group monomorphism into $\GL(4,\C)$. A direct computation also proves that the following identities are satisfied.
	\begin{enumerate}
		\item For every $A \in M_2(\HH)$: $\psi(A^*) = \psi(A)^*$.
		\item For every $Z, W \in M_2(\C)$: $\psi(Z^\top - W^* j) = \psi(Z + Wj)^\top$.
	\end{enumerate}
	Let us now consider the matrices $J_2 = \psi(I_2 j)$ and $K_{1,1} = \psi(I_{1,1})$. Then, it is easy to see that
	\[
		\psi(\GL(2,\HH)) = \{ M \in \GL(4,\C) : M J_2 = J_2 \overline{M} \}.
	\]
	Using the previous facts we conclude that
	\[
		\psi(\Spe(1,1)) = \{ M \in \GL(4,\C) : M^* K_{1,1} M = K_{1,1}, \;
							M J_2 = J_2 \overline{M} \}.
	\]
	Next, we observe that the two conditions on $M \in \GL(4,\C)$ that define $\psi(\Spe(1,1))$ are equivalent to 
	\[
		M^* K_{1,1} M = K_{1,1}, \quad M^\top K_{1,1} J_2 M = K_{1,1} J_2.
	\]
	If we define the matrix
	\[
		K_{1,i} = 
		\begin{pmatrix}
			\diag(1,i) & 0 \\
			0 & \diag(1,i)
		\end{pmatrix},
	\]
	then it is easy to check that $K_{1,i}^2 = K_{1,1}$ and that all three matrices $J_2$, $K_{1,1}$ and $K_{1,i}$ pairwise commute. In particular, we have
	\[
		K_{1,i}^* K_{1,1} K_{1,i} = K_{1,1}.
	\]
	It follows that if we consider the isomorphism $\rho : \GL(4,\C) \rightarrow \GL(4,\C)$ given by 
	\[
		\rho(M) = K_{1,i} M K_{1,i}^{-1},
	\]
	then $\rho$ maps $\psi(\Spe(1,1))$ onto the subgroup
	\begin{equation}\label{eq:Sp(1,1)Helgason}
		\widehat{\Spe}(1,1) = 
			\{M \in \GL(4,\C) : M^* K_{1,1} M = K_{1,1}, \; M^\top J_2 M = J_2\}.
	\end{equation}
	The latter yields the definition of $\Spe(1,1)$ in its realization as a group of complex matrices (see~\cite[page~445]{Helgason}). In particular, the previous discussion shows that
	\[
		\rho \circ \psi : \Spe(1,1) \rightarrow \widehat{\Spe}(1,1),
	\]
	is an isomorphism of Lie groups. This allows us to apply the theory of Riemannian symmetric spaces developed in \cite{Helgason} to the realization \eqref{eq:BasSymmQuotient} of the unit ball obtained above from the classical M\"obius transformations.
\end{remark}

As a consequence of the discussion in Remark~\ref{rmk:Sp(1,1)vsHelgason} we find that $(\B,\widehat{g})$ is a Riemannian symmetric space. We refer to \cite[Table~V,~p.~518]{Helgason} for further details on the notation below.

\begin{corollary}\label{cor:BRiemSymm}
	The Riemannian manifold $(\B,\widehat{g})$ is the Riemannian symmetric space of type CII associated to the symmetric pair $(\Spe(1,1), \Spe(1) \times \Spe(1))$ and the automorphism of $\Spe(1,1)$ given by $A \mapsto I_{1,1} A I_{1,1}$.
\end{corollary}
\begin{proof}
	From \cite[Table~V,~p.~518]{Helgason} we know that the Riemannian symmetric space of type CII is given by $\widehat{\Spe}(1,1)$ (see~\ref{eq:Sp(1,1)Helgason}) with the automorphism $M \mapsto K_{1,1} M K_{1,1}$. Using the isomorphism $\rho\circ \psi$ from Remark~\ref{rmk:Sp(1,1)vsHelgason}, the result follows from a straightforward computation using that $\psi(I_{1,1}) = K_{1,1}$.
\end{proof}

Since $(\B,\widehat{g})$ is a Riemannian symmetric space coming from the group $\Spe(1,1)$, the identity component $I_0(\B, \widehat{g})$ of its isometry group is naturally isomorphic to $\Spe(1,1)/\{\pm I_2\}$, where we have taken quotient by the center. As the subgroup $\{\pm I_2\}$ is finite, replacing $I_0(\B, \widehat{g})$ with $\Spe(1,1)$ does not affect the Lie theory claims made so far and in the rest of this work. This follows from \cite[Chapter~IV,~Proposition~3.4]{Helgason}, which allows to realize Riemannian symmetric spaces with Riemannian symmetric pairs not necessarily given by the identity component of the isometry group. We also refer to \cite{BisiGentiliMobius,QB-SliceKahler} for more elementary approaches to this fact.

We will now apply the results from subsection~\ref{subsec:RiemSymmSpaces} to obtain further Lie theory and geometric invariants associated to the Riemannian symmetric space $(\B,\widehat{g})$.

According to \eqref{eq:MobiusClassical}, the matrix $I_{1,1}$, which clearly belongs to $\Spe(1,1)$, yields the isometry $q \mapsto -q$. Hence, the automorphism $\sigma : \Spe(1,1) \rightarrow \Spe(1,1)$ corresponding, in this case, to Proposition~\ref{prop:RiemSymmProperties}(2) is given by
\[
	\sigma\bigg(
		\begin{pmatrix}
			a & c \\
			b & d
		\end{pmatrix}
		\bigg) = I_{1,1} 
		\begin{pmatrix}
			a & c \\
			b & d
		\end{pmatrix}
		I_{1,1} 
		= \begin{pmatrix*}[r]
			a & -c \\
			-b & d
		\end{pmatrix*}.
\]
The fixed point set of $\sigma$ is the subgroup $\Spe(1) \times \Spe(1)$. It follows that \eqref{eq:BasSymmQuotient} yields the realization of $\B$ as homogeneous space stated in Proposition~\ref{prop:RiemSymmProperties}(1).

On the other hand, the tangent space at the identity $I_2$ of the Lie group $\Spe(1,1)$ is given by
\begin{align*}
	\spe(1,1) &= \{ X \in M_2(\HH) : X^* I_{1,1} + I_{1,1} X = 0 \} \\
		&= \Bigg\{
			\begin{pmatrix}
				p & \overline{a} \\
				a & q
			\end{pmatrix} : p,q \in \im(\HH), a \in \HH
			\Bigg\},
\end{align*}
which follows from a simple computation. Furthermore, the arguments from \cite[Example,~p.110]{Helgason} and \cite[Chapter~II,~Theorem~2.1]{Helgason} show that the Lie algebra structure of $\spe(1,1)$ is given by the Lie brackets $[X,Y] = XY - YX$, the commutator of matrices. This fact can also be proved using the algebra homomorphism $\psi$ from Remark~\ref{rmk:Sp(1,1)vsHelgason}.

Note that the expression defining $\sigma$ naturally extends to a real linear map on the space of matrices $M_2(\HH)$, and so the differential of $\sigma$ at the identity $I_2$ yields the automorphism of Lie algebras $\dif \sigma_{I_2} : \spe(1,1) \rightarrow \spe(1,1)$ given by the same formula
\[
	\dif \sigma_{I_2}
		\bigg(
		\begin{pmatrix}
			a & c \\
			b & d
		\end{pmatrix}
		\bigg) 
		= \begin{pmatrix*}[r]
			a & -c \\
			-b & d
		\end{pmatrix*}.
\]
Hence, the subspaces considered in Proposition~\ref{prop:RiemSymmProperties}(3) corresponding to the Riemannian symmetric space $(\B,\widehat{g})$ are given as follows
\begin{align*}
	\frakk_\B &= \spe(1) \times \spe(1) 
			= \{\diag(p,q) : p,q \in \im(\HH) \}  \\
	\frakm_\B &= 
		\Bigg\{
		\begin{pmatrix}
			0 & \overline{q} \\
			q & 0
		\end{pmatrix} : q \in \HH
		\Bigg\}.
\end{align*}
We will use this notation in the rest of this work.

To complete our description of the Riemannian symmetric space $(\B,\widehat{g})$, we compute the exponential map of $\Spe(1,1)$ in the next result.

\begin{proposition}\label{prop:expSp(1,1)}
	The Lie group exponential map $\exp : \spe(1,1) \rightarrow \Spe(1,1)$ is given by the usual matrix exponentiation. In other words, we have
	\[
		\exp(tX) = \sum_{k=0}^{+\infty} \frac{t^k X^k}{k!},
	\]
	for every $X \in \spe(1,1)$. In particular, the exponential map restricted to $\frakm_\B$ can be expressed through the formula
	\[
		\exp
		\bigg(t
			\begin{pmatrix}
				0 & \overline{u} \\
				u & 0
			\end{pmatrix}
		\bigg) 
		= 
			\begin{pmatrix*}[r]
				\cosh(t) & \sinh(t)\overline{u} \\
				\sinh(t) u & \cosh(t)
			\end{pmatrix*} 
	\]
	where $t \in \R$ and $u \in \Spe(1)$.
\end{proposition}
\begin{proof}
	If we denote by $\gamma(t)$ the right hand-side of the first formula in the statement, then it is a standard exercise to show that $\gamma(s + t) = \gamma(s) \gamma(t)$. Hence, the first claim follows. The second formula is consequence of a straightforward computation. For this note that every $q \in \HH$ can be written $q = tu$ for some $t \in \R$ and $u \in \Spe(1)$.
\end{proof}

As a consequence of Propositions~\ref{prop:RiemSymmProperties}(4) and \ref{prop:expSp(1,1)}, we conclude that the geodesics in $(\B,\widehat{g})$ through the origin are given by
\begin{equation}\label{eq:BSymmGeodesics}
	\gamma_u(t) 
		= 0\cdot
		\exp
		\bigg(t
		\begin{pmatrix}
			0 & \overline{u} \\
			u & 0
		\end{pmatrix}
		\bigg)
		= 0\cdot
		\begin{pmatrix*}[r]
			\cosh(t) & \sinh(t)\overline{u} \\
			\sinh(t) u & \cosh(t)
		\end{pmatrix*}
		= \tanh(t)u,
\end{equation}
where $u \in \Spe(1)$. Note that $\gamma_u'(0) = u$, and so $\gamma_u$ is a geodesic parametrized by arc-length.

It will be useful to describe, with respect to the $\Spe(1,1)$-action, all the orbits for the one-parameter subgroups associated to $\frakm_\B$. This is done in the next result, which follows from Proposition~\ref{prop:expSp(1,1)} and a straightforward computation.

\begin{proposition}\label{prop:HyperOrbits}
	For every $u_0 \in \Spe(1)$, the one parameter subgroup of $\Spe(1,1)$ defined by 
	\[
		t \mapsto \exp 
		\bigg(t
		\begin{pmatrix}
			0 & \overline{u}_0 \\
			u_0 & 0
		\end{pmatrix}
		\bigg)
	\]
	has orbit passing through $a \in \B$ given by the set of points of the form
	\[
		(1 + \tanh(t)a \overline{u}_0)^{-1}
			(a + \tanh(t) u_0),
	\]
	where $t \in \R$.
\end{proposition}

Proposition~\ref{prop:G=KM} and Remark~\ref{rmk:G=KM} imply the next decomposition of $\Spe(1,1)$. From the discussion so far, one can say that such decomposition is induced by the geometry of the Riemannian symmetric space $(\B,\widehat{g})$. As a matter of fact, we provide a proof based on the geometric properties of $(\B,\widehat{g})$. This allows us to introduce some arguments that will be useful later on.

\begin{proposition}\label{prop:Sp(1,1)=KB}
	The map $\widehat{\pi} : \Spe(1) \times \Spe(1) \times \frakm_\B \rightarrow \Spe(1,1)$ given by
	\[
		(u,v,X) \longmapsto \diag(u,v) \exp(X),
	\]
	is a diffeomorphism.
\end{proposition}
\begin{proof}
	First note that $\widehat{\pi}$ is clearly a smooth map.
	
	Let $A \in \Spe(1,1)$ be given, and consider $a = 0\cdot A \in \B$. Since $(\B,\widehat{g})$ is a complete Riemannian manifold, there is a geodesic passing through both $0$ and $a$. Hence, from the previous discussion on geodesics for $(\B,\widehat{g})$, there exists $X \in \frakm_\B$ such that $a = 0 \cdot \exp(X)$. We conclude that $A \exp(X)^{-1}$ fixes the origin and so, by the description of the isotropy of the origin, there exist $u,v \in \Spe(1)$ such that
	\[
		A \exp(X)^{-1} = \diag(u,v).
	\]
	This proves that $\widehat{\pi}(u,v,X) = A$. Then, $\widehat{\pi}$ is a surjective map.
	
	Let us now assume that $A = \diag(u_1,v_1) \exp(X_1) = \diag(u_2,v_2) \exp(X_2)$, where $u_j, v_j \in \Spe(1)$ and $X_j \in \frakm_\B$ ($j=1,2$). If we apply $A$ to the origin, then we obtain
	\[
		0\cdot \exp(X_1) = 0\cdot \exp(X_2) = a,
	\]
	for some $a \in \B$. We can write
	\[
		X_j = t_j
			\begin{pmatrix}
				0 & \overline{w}_j \\
				w_j & 0
			\end{pmatrix},
	\]
	for some $t_j \in \R$ and $w_j \in \Spe(1)$ ($j=1,2$). From our assumptions, and equation~\eqref{eq:BSymmGeodesics}, it follows that
	\[
		\tanh(t_1) w_1 = a = \tanh(t_2) w_2.
	\]
	By taking absolute value on this identity, we obtain $t_1 = \pm t_2$. Hence, we conclude that either $t_1 = t_2$ and $w_1 = w_2$, or $t_1 = - t_2$ and $w_1 = -w_2$. In any case, we obtain $X_1 = X_2$. Then, from both expressions for $A$ above we also conclude that $(u_1,v_1) = (u_2,v_2)$. This proves the injectivity of $\widehat{\pi}$.
	
	It remains to compute $\widehat{\pi}^{-1}$ and verify its smoothness. To achieve this, we will consider the map $f : \frakm_\B \rightarrow \B$ given by 
	\[
		f(X) = 0\cdot \exp(X) = F_{\exp(X)}(0).
	\]
	We claim that $f$ is a diffeomorphism. This map is clearly smooth and the arguments above show that it is bijective. Since the dimension of both $\frakm_\B$ and $\B$ is $4$, it is enough to prove that the differential of $f$ is injective at every point. First, we observe that \eqref{eq:BSymmGeodesics} and the remark that follows imply 
	\[
		\dif f_0
		\bigg(
		\begin{pmatrix}
			0 & \overline{q}  \\
			q & 0
		\end{pmatrix}
		\bigg) = q,
	\]
	for every $q \in \HH$, and thus the required injectivity holds at $0$. Let us now consider the map $g : (0,+\infty) \times \Spe(1) \rightarrow \frakm_\B \setminus \{0\}$ given by 
	\[
		(t,w) \mapsto 
		t \begin{pmatrix}
			0 & \overline{w} \\
			w & 0
		\end{pmatrix},
	\]
	which is clearly a diffeomorphism; in fact, it defines the spherical coordinates of $\frakm_\B \setminus \{0\}$. From the previous arguments, we have the identity $f \circ g(t,w) = \tanh(t) w$ and so, for given $(t_0, w_0) \in (0,+\infty) \times \Spe(1)$, it is easy to compute
	\[
		\dif\;(f \circ g)_{(t_0,w_0)}(t,\alpha) 
			= \cosh(t_0)^{-2} t w_0 + \tanh(t_0) \alpha,
	\]
	for every $t \in \R$ and $\alpha \in T_{w_0} \Spe(1)$. If we identify $\HH \simeq \R^4$, then $\Spe(1)$ is the unit sphere and so the tangent space $T_{w_0} \Spe(1)$ is the orthogonal complement of $w_0$. We conclude that $\dif\;(f \circ g)_{(t_0,w_0)}(t,\alpha) = 0$
	implies $(t,\alpha) = 0$. This completes the proof that $f$ is a diffeomorphism.
	
	Let us consider $A \in \Spe(1,1)$. By the surjectivity of $\widehat{\pi}$, there exist 
	$(u,v,X) \in \Spe(1) \times \Spe(1) \times \frakm_\B$ such that $A = \widehat{\pi}(u,v,X)$, i.e.~we have
	\[
		A = \diag(u,v) \exp(X).
	\]
	Applying this expression to $0 \in \B$ we obtain $0\cdot A = 0\cdot \exp(X)$. From the definition of $f$ we conclude that $X = f^{-1}(0 \cdot A)$, and we also have
	\[
		\diag(u,v) = A \exp(-f^{-1}(0\cdot A)).
	\]
	If we identify a pair $(u,v)$ with the matrix $\diag(u,v)$, then the previous identities yield 
	\[
		\widehat{\pi}^{-1}(A) = (A \exp(-f^{-1}(0\cdot A)), f^{-1}(0\cdot A)),
	\]
	for every $A \in \Spe(1,1)$. Since the action of $\Spe(1,1)$ on $\B$, the map $f^{-1}$ and the product of matrices are all smooth, it follows that $\widehat{\pi}^{-1}$ is smooth as well.	
\end{proof}

To finish this subsection, we now apply Propositions~\ref{prop:RiemSymmProperties} and \ref{prop:RiemSymmCentralizers}, together with Remark~\ref{rmk:RiemSymmCentralizers}, to the Riemannian symmetric space $(\B,\widehat{g})$ (see also~\eqref{eq:BasSymmQuotient}). The conclusion of this result is two-fold: to explicitly write $\B$ as quotient of $\Spe(1,1)$ coming from the action by classical M\"obius transformations and to describe $I_0(\B,\widehat{g})$ in terms of centralizers in the group $\Spe(1,1)$.

\begin{theorem}\label{thm:RiemSymmBCentralizers}
	The right $\Spe(1,1)$-action on $\B$ given by \eqref{eq:MobiusClassical} is isometric on $(\B,\widehat{g})$, yields the identity component isometry group $I_0(\B,\widehat{g})$ and provides the homogeneous Riemannian symmetric space realization given by
	\begin{align}\label{eq:BasSymmQuotientProposition}
		(\Spe(1) \times \Spe(1)) \backslash \Spe(1,1) &\xlongrightarrow{\sim} \B \\
		(\Spe(1) \times \Spe(1)) A &\longmapsto F_A(0). \notag
	\end{align}
	Furthermore, the group $I_0(\B,\widehat{g})$ is obtained from the symmetries of $\Spe(1,1)$, that descend in the quotient to $\B$, from the following list 
	\begin{enumerate}
		\item $\{R_A : A \in C_{\Spe(1,1)}(\{I_2\}) \}$, and
		\item $\{L_A : A \in C_{\Spe(1,1)}(\Spe(1) \times \Spe(1)) \}$,
	\end{enumerate}
	where we have $C_{\Spe(1,1)}(\{I_2\}) = \Spe(1,1)$ and $C_{\Spe(1,1)}(\Spe(1) \times \Spe(1)) = \{\pm I_2\}$.
\end{theorem}

\section{Slice regularity and Riemannian geometry on $\B$}\label{sec:SliceRiemB}

\subsection{Slice regular functions on $\B$}\label{subsec:sliceregular}
We recall and discuss some properties of slice regular functions of a quaternionic variable defined on $\B$. We refer to \cite{ColomboSabadiniStruppaFunctionalBook,GentiliStoppatoStruppa2ndEd} for further details.

A function $f : \B \rightarrow \HH$ is called slice regular if it is smooth and satisfies
\[
	\frac{1}{2} \bigg(\frac{\partial}{\partial x} 
		+ I \frac{\partial}{\partial y}\bigg) f (x+yI) = 0,
\]
for every $x,y \in \R$ and $I \in \Sbb$ such that $x + y I \in \B$. From now on, $\Sbb$ will denote the $2$-dimensional sphere of imaginary units in $\HH$. The previous condition is precisely the requirement that $f$ be holomorphic when restricted to the complex slice $\C_I = \R \oplus \R I$, for every $I \in \Sbb$.

Every slice regular function $f$ on $\B$ admits a power series representation 
\[
	f(q) = \sum_{n=0}^{+\infty} q^n a_n,
\]
where $a_n \in \HH$, for all $n$, that converges uniformly on compact subsets of $\B$. Furthermore, this property completely characterizes slice regular functions on $\B$.

The family of slice regular functions fails to be closed under pointwise product. However, it is possible to define a slice regular product, also called $*$-product, that thus yields a real algebra structure to such family. In terms of power series as above, the $*$-product is defined, for slice regular functions on $\B$, by the formula
\[
	\bigg(\sum_{n=0}^{+\infty} q^n a_n \bigg) * 
		\bigg(\sum_{n=0}^{+\infty} q^n b_n \bigg) =
			\sum_{n=0}^{+\infty} q^n \sum_{k=0}^n a_k b_{n-k}. 
\]
One can further introduce the regular conjugate $f^c$ and the symmetrization $f^s$ of $f$ as above with the following formulas
\[
	f^c(q) = \sum_{n=0}^{+\infty} q^n \overline{a}_n, \quad
	f^s = f * f^c = f^c * f,
\]
respectively. With this notation, we are able to define
\[
	f^{-*} = \frac{1}{f^s} f^c,
\]
which turns out be the inverse of $f$ with respect to the $*$-product. More precisely, we have
\[
	f * f^{-*} = f^{-*} * f = 1
\]
the constant map $1$, on the complement of a suitable zero set.

For every $a,b \in \HH$, we consider the regular map $\ell_{a,b}(q) = qa + b$, where $q \in \B$. Then, for every $A \in \Spe(1,1)$ with matrix entries given by
\[
	A = \begin{pmatrix}
		a & b \\
		c & d
	\end{pmatrix},
\]
the regular M\"obius transformation associated to $A$ is the map $\B \rightarrow \HH$ defined by
\[
	\cF_A = \ell_{c,d}^{-*} * \ell_{a,b}.
\]
It is well known that $\cF_A$ maps $\B$ onto $\B$ diffeomorphically. Every such map $\cF_A$ will be called a \textbf{regular M\"obius transformation} and the family of all these maps is denoted by $\cM(\B)$.

It is natural to ask about the conditions needed for classical M\"obius transformation, elements of $\Mbb(\B)$, to be regular. This is answered in the next result.

\begin{proposition}\label{prop:MobiusClassicalRegular}
	The intersection $\Mbb(\B) \cap \cM(\B)$ consists of the transformations $F_{a,u} : \B \rightarrow \B$ defined by
	\[
		F_{a,u}(q) = (1 - qa)^{-1} (q - a) u,
	\]
	where $a \in (-1,1)$ and $u \in \Spe(1)$.
\end{proposition}
\begin{proof}
	For $a \in (-1,1)$ and $u \in \Spe(1)$, it is well known that the classical M\"obius transformation $F_{a,u}$ in the statement is regular because it can be rewritten as
	\[
		F_{a,u} = R_u \circ (\ell_{-a,1}^{-*}*\ell_{1,-a}),
	\]
	where $R_u(q) = qu$. It follows that $F_{a,u}$ belongs to $\cM(\B)$ as well.
	
	Let us now consider a slice regular map $F$ on $\B$ that also belongs to $\Mbb(\B)$. Then, there exist $a \in \B$ and $u,v \in \Spe(1)$ such that (see~\cite{BisiStoppatoMobius})
	\[
		F(q) = v (1 - q\overline{a})^{-1} (q - a)u,
	\]
	for every $q \in \B$. Hence, with the use of a geometric series and some manipulations we conclude that
	\begin{align*}
		F(q) &= v\sum_{n=0}^{+\infty}  (q\overline{a})^n (q - a)u 
			= -vau + v\sum_{n=0}^{+\infty} (q\overline{a})^n q u
				- v \sum_{n=1}^{+\infty} (q \overline{a})^n a u \\
			&= -vau + v \sum_{n=0}^{+\infty}(q\overline{a})^n q (1 - |a|^2)u.
	\end{align*}
	By adding $vau$ and then multiplying on the right by $(1 - |a|^2)^{-1} \overline{u}$, we conclude that the function defined by
	\[
		F_1(q) = v \sum_{n=0}^{+\infty}(q\overline{a})^n q
	\]
	is slice regular. For every $x,y \in \R$ and $I \in \Sbb$ such that $x + yI \in \B$, using a Leibniz rule with respect to the product in $\HH$, it is straightforward to compute the following partial derivatives
	\begin{align*}
		\frac{\partial F_1}{\partial x}&(x+yI) = \\
			=&\; v +
			v \sum_{n=1}^{+\infty} \bigg(
				\sum_{k=0}^{n-1}\big((x+yI)\overline{a}\big)^k \overline{a}
					\big((x+yI)\overline{a}\big)^{n-1-k}(x+yI) \\
					&+
					\big((x+yI)\overline{a}\big)^n
					\bigg) \\
		\frac{\partial F_1}{\partial y}&(x+yI) = \\
			=&\; vI +
			v \sum_{n=1}^{+\infty} \bigg(
				\sum_{k=0}^{n-1}\big((x+yI)\overline{a}\big)^k I\overline{a}
					\big((x+yI)\overline{a}\big)^{n-1-k}(x+yI) \\
					&+
					\big((x+yI)\overline{a}\big)^n I
					\bigg).
	\end{align*}
	Taking $y = 0$, we conclude that, for every $x \in (-1,1)$ and $I \in \Sbb$, we have
	\begin{align*}
		\frac{\partial F_1}{\partial x}(x) 
			&= v +
				v \sum_{n=1}^{+\infty} (n+1) x^n \overline{a}^n \\
		I\frac{\partial F_1}{\partial y}(x+yI) 
			&= IvI +
				Iv \sum_{n=1}^{+\infty} x^n
				\sum_{k=0}^{n}\overline{a}^k I
					\overline{a}^{n-k}.		
	\end{align*}
	The slice regularity of $F_1$ implies, for $x = 0$, that $v = -IvI$ for every $I \in \Sbb$ and so that $v \in \R$. Since $v \in \Spe(1)$, we conclude that $v = \pm 1$. For this value of $v$, we use again the slice regularity condition of $F_1$ to conclude that the order $1$ coefficient of the sum of the two series above vanishes. This yields the identity
	\[
		2 \overline{a} = -I(I\overline{a} + \overline{a}I),
	\]
	for every $I \in \Sbb$. Hence, $a \in \R$ and so $a \in (-1,1)$ because $a \in \B$. We conclude that $F = F_{a,\pm u}$, with the required conditions on $a$ and $\pm u$. This completes the proof.
\end{proof}

The maps of $\B$ belonging to $\Mbb(\B) \cap \cM(\B)$ can be described as classical M\"obius transformations obtained from subgroups of $\Spe(1,1)$. Let us consider the following subgroups
\begin{align*}
	\OO(1,1) &= \Spe(1,1) \cap M_2(\R) \\
	\SO(1,1) &= \{ A \in \OO(1,1) : \det(A) = 1 \} \\
	\SO_0(1,1) &= \biggl\{
	\begin{pmatrix*}[r]
		\cosh(t) & \sinh(t) \\
		\sinh(t) & \cosh(t) 
	\end{pmatrix*} : t \in \R
	\biggr\}.
\end{align*}
It is clear that $\SO_0(1,1) \subset \SO(1,1) \subset \OO(1,1)$. It follows from a simple inspection (see also \cite[Chapter~X,~Lemma~2.4]{Helgason}) that $\SO_0(1,1)$ is the identity component of both $\SO(1,1)$ and $\OO(1,1)$. Furthermore, we have the following subgroup class decompositions
\begin{align*}
	\SO(1,1) &= \SO_0(1,1) \bigcup (-\SO_0(1,1)), \\
	\OO(1,1) &= \SO(1,1) \bigcup \SO(1,1)I_{1,1}, \\
				&= \SO_0(1,1) \bigcup \SO_0(1,1)I_{1,1}
					\bigcup 
					(-\SO_0(1,1)) \bigcup (-\SO_0(1,1)I_{1,1})
\end{align*}
We will consider $\Spe(1)$ acting on the right by the transformations
\[
	q\cdot u = qu,
\]
for every $q \in \B$ and $u \in \Spe(1)$. After a straightforward computation of the $\SO_0(1,1)$-action by classical M\"obius transformations, Proposition~\ref{prop:MobiusClassicalRegular} implies that the actions of $\SO_0(1,1)$ and $\Spe(1)$ yield all the regular classical M\"obius transformations. This proves the next result.

\begin{corollary}\label{cor:MobiusClassicalRegular}
	The set $\Mbb(\B) \cap \cM(\B)$ is the collection of transformations given by the following actions.
	\begin{enumerate}
		\item The right $\SO_0(1,1)$-action on $\B$ by classical M\"obius transformations.
		\item The right $\Spe(1)$-action on $\B$ given by $q\cdot u = qu$.
	\end{enumerate}
\end{corollary}

\subsection{The slice Riemannian metric on $\B$}\label{subsec:sliceRiem}
We have introduced in \cite{QB-SliceKahler} geometric structures on the unit ball $\B$ that are closely related to regular M\"obius transformations on this domain. We recall the basic definitions and properties, while referring to \cite{QB-SliceKahler} for further details. 

For every $q \in \B$ let us consider the form given by
\begin{align*}
	h_q : \HH \times \HH &\rightarrow \HH  \\
	h_q(\alpha,\beta) &= 
		\frac{(1-q^2)^{-1} (\alpha - \alpha q\alpha) 
			(\overline{\beta - \beta q \beta}) (1-\overline{q}^2)^{-1}}{(1-|q|^2)^2}
\end{align*}
where $q \in \B$. Since the tangent space of $\B$ at every point is $\HH$, this yields a smooth tensor $h$. Furthermore, $h$ is an $\HH$-valued, $\R$-bilinear, Hermitian symmetric, positive definite form. In particular, the following properties are satisfied (see~\cite{QB-SliceKahler})
\begin{enumerate}
	\item ($\HH$-Hermitian symmetry): at every $q \in \B$ and for every $\alpha, \beta \in \HH$: $h_q(\alpha,\beta) = \overline{h_q(\beta,\alpha)}$.
	\item (positive definiteness): at every $q \in \B$ and for every $\alpha \in \HH$: $h_q(\alpha,\alpha) \geq 0$ and it vanishes if and only if $\alpha = 0$.
\end{enumerate}
We call $h$ the slice Hermitian metric on $\B$.

The properties just stated for $h$ yield two additional geometric structures on $\B$: the slice Riemannian metric $g$ on $\B$ and the slice K\"ahler form $\omega$ defined as the real and (quaternionic) imaginary parts of $h$, respectively. Our main reason to consider these geometric structures is that they relate nicely to regular M\"obius transformations. In particular, they satisfy the following properties.
\begin{enumerate}
	\item (Slice invariance): for every $u \in \Spe(1)$, the map $C_u : \B \rightarrow \B$ given by $C_u(q) = \overline{u} q u$ is an isometry of $g$.
	\item (Slice hyperbolicity): for every $I \in \Sbb$, the restriction of $g$ to the complex slice $\C_I$ (both on the base point and the tangent vectors) yields the usual complex hyperbolic metric on the complex unit disk $\C_I \cap \B$.
	\item (Slice regular invariance): for every $q \in \B$ and $\cF \in \cM(\B)$ such that $\cF(q) = 0$ we have
	\[
		g_q(\alpha,\beta) = g_0(\dif \cF_q(\alpha), \dif \cF_q(\beta)) 
			= \dif \cF_q(\alpha) \overline{\dif \cF_q(\beta)},
	\]
	for every $\alpha, \beta \in \HH$.
\end{enumerate}
These claims have been proved in \cite{QB-SliceKahler}, where we have also obtained the formula
\[
	g_q(\alpha,\beta) = 
		\frac{\re((\alpha - \alpha q\alpha) 
			(\overline{\beta - \beta q \beta})}
			{|1-q^2|^2 (1-|q|^2)^2},
\]
which holds for every $q \in \B$ and $\alpha, \beta \in \HH$. As a matter of fact, this expression implies all three properties stated above for $g$. It also shows that our slice Riemannian metric $g$ is precisely the Riemannian metric obtained in \cite{ArcozziSarfatti} from the pseudo-hyperbolic distance induced by a quaternionic Hardy space in $\B$. For further details on this claim, we refer to \cite{QB-SliceKahler}. Hence, we will freely use in the rest of this work that the slice Riemannian metric $g$ on $\B$ satisfies the properties established in \cite{ArcozziSarfatti}. In particular, the next result is a consequence of that reference.

\begin{proposition}[Arcozzi-Sarfatti]\label{prop:Arcozzi-Sarfatti}
	The slice Riemannian metric $g$ on $\B$ satisfies the following properties.
	\begin{itemize}
		\item The complex unit disk $\C_I \cap \B$ is totally geodesic for every $I \in \Sbb$. In particular, the rays in $\B$ through the origin are geodesics.
		\item The isometry group of $(\B,g)$ is generated by the following maps.
		\begin{enumerate}
			\item The signed slice symmetries given by the left $\{\pm 1\} \times \Spe(1)$-action on $\B$ defined by $(\epsilon, u) \cdot q = \epsilon u q\overline{u}$, where $\epsilon \in \{\pm 1\}$ and $u \in \Spe(1)$.
			\item The (regular and) classical M\"obius transformations given by the right $\SO_0(1,1)$-action.
			\item The involution $q \mapsto \overline{q}$, whose fixed point set is precisely the real axis.
		\end{enumerate}
	\end{itemize}
\end{proposition}
\begin{proof}
	The first part follows from \cite[Lemma~4.4]{ArcozziSarfatti}. For the second claim we use \cite[Subsection~4.1]{ArcozziSarfatti}. We observe that the isometry group of $\Sbb$ is given precisely by the right $\{\pm 1\} \times \Spe(1)$-action from our statement.
\end{proof}

We now explicitly describe the isometry group of $(\B,g)$. First, we will consider the matrices
\[
	H(t) = 
		\begin{pmatrix}
			\cosh(t) & \sinh(t) \\
			\sinh(t) & \cosh(t)
		\end{pmatrix},
\]
for every $t \in \R$. In other words, the matrices $H(t)$ ($t \in \R$) yield all the elements of $\SO_0(1,1)$. We will use this notation from now on without further mention. Next, we consider the following product operation
\[
	(\epsilon_1, H(t_1)) \star (\epsilon_2, H(t_2)) 
		= (\epsilon_1 \epsilon_2, H(\epsilon_2 t_1 + t_2)), 
\]
for every $\epsilon_1, \epsilon_2 \in \{\pm 1\}$ and $t_1, t_2 \in \R$. It is straightforward to see the $\star$ defines a group structure on $\Z_2 \times \SO_0(1,1)$, where $\Z_2 = \{\pm 1\}$ carries the natural multiplication operation. The Lie group thus obtained is called the semi-direct product of $\Z_2$ and $\SO_0(1,1)$ and it will be denoted by $\Z_2 \ltimes \SO_0(1,1)$.

\begin{theorem}\label{thm:Iso(B,g)}
	Let us consider the maps given by
	\[
		(u,\epsilon_1, H(t),\epsilon_2) \cdot q = 
			\begin{cases}
				\epsilon_1 u (1 + \tanh(t)q)^{-1} 
				(q + \tanh(t))\overline{u}, & \text{ if } \epsilon_2 = +1, \\
				\epsilon_1 u (1 + \tanh(t) \overline{q})^{-1} 
				(\overline{q} + \tanh(t))\overline{u}, & \text{ if } \epsilon_2 = -1.
			\end{cases}
	\]
	Then, these maps define a left action of the group $\Spe(1) \times (\Z_2 \ltimes \SO_0(1,1)) \times \Z_2$ on $\B$. Furthermore, this action realizes the full isometry group of $(\B,g)$.
\end{theorem}
\begin{proof}
	Note that the factors of the (set) product $\Spe(1) \times \Z_2 \times \SO_0(1,1) \times \Z_2$ are groups. Furthermore, if we restrict the value of $(u,\epsilon_1, H(t),\epsilon_2)$ in the set $\Spe(1) \times \Z_2 \times \SO_0(1,1) \times \Z_2$ so that all but a single entry is the identity element, then we obtain actions of corresponding factors. More precisely, this yields actions of the groups $\Spe(1)$, $\Z_2$ (second factor), $\SO_0(1,1)$ and $\Z_2$ (fourth factor). 
	
	It is easy to check that all four actions commute pairwise except for the case of the $\Z_2$-action of the second factor and the $\SO_0(1,1)$-action. However, the maps in the statement corresponding to values $(1,\epsilon,H(t),1)$ yield an action of the group $\Z_2 \ltimes \SO_0(1,1)$ with the semi-direct product $\star$ defined above. This can be verified with a simple computation.
	
	Hence, the maps in the statement yield actions of $\Spe(1)$, $\Z_2 \ltimes \SO_0(1,1)$ and $\Z_2$ (last factor in the product) that pairwise commute. We thus obtain an action of $\Spe(1) \times (\Z_2 \ltimes \SO_0(1,1)) \times \Z_2$ with the product (three factors) group structure.
	
	The set of maps obtained from the action of $\Spe(1) \times (\Z_2 \ltimes \SO_0(1,1)) \times \Z_2$ is clearly the group generated by the isometries of $(\B,g)$ listed in Proposition~\ref{prop:Arcozzi-Sarfatti}. And so, the latter implies that the action of this group realizes the isometry group of $(\B,g)$.
\end{proof}

\begin{remark}\label{rmk:Iso(B,g)}
	The action considered in Theorem~\ref{thm:Iso(B,g)} is not effective, i.e.~there exist non-trivial elements in the group acting trivially. More precisely, the subgroup of those elements of $\Spe(1) \times (\Z_2 \ltimes \SO_0(1,1)) \times \Z_2$ that act trivially (as the identity transformation) is given by $\{ (\pm 1, 1, I_2, 1) \}$. We conclude that the isometry group of $(\B,g)$ is isomorphic to
	\[
		\Spe(1)/\{\pm 1\} \times (\Z_2 \ltimes \SO_0(1,1)) \times \Z_2 
			\simeq \SO(3) \times (\Z_2 \ltimes \R) \times \Z_2,
	\]
	where $\Z_2 \ltimes \R$ has the semi-direct group structure given by the product 
	\[
		(\epsilon_1, t_1) \star (\epsilon_2, t_2) 
			= (\epsilon_1 \epsilon_2, \epsilon_2 t_1 + t_2) 
	\]
	for every $\epsilon_1, \epsilon_2 \in \{\pm 1\}$ and $t_1, t_2 \in \R$. This uses the fact that the map $t \mapsto H(t)$ defines an isomorphism of Lie groups between $\R$ and $\SO_0(1,1)$.
\end{remark}

\subsection{Lie theoretic symmetries of the slice Riemannian metric on $\B$} \label{subsec:sliceRiemLie}
As noted in Theorem~\ref{thm:RiemSymmBCentralizers}, the Riemannian symmetric space structure of $\B$ yields a realization of this space as quotient, on the right, of $\Spe(1,1)$ by the subgroup $\Spe(1) \times \Spe(1)$ (see~equation~\ref{eq:BasSymmQuotientProposition}). A corresponding, nevertheless different, realization of $\B$ can be obtained from the regular M\"obius transformations introduced in subsection~\ref{subsec:sliceregular}. More precisely, the next result has been proved in \cite{QB-cM(B)-asmanifold}. Note that the double coset considered in the statement does define a manifold as a consequence, in part, of the fact that both subgroups $(\Spe(1) \times \{1\})$ and $\Spe(1) I_2$ are compact. We refer to \cite{QB-cM(B)-asmanifold} for further details.

\begin{theorem}[Quiroga-Barranco \cite{QB-cM(B)-asmanifold}]
\label{thm:SliceRegularBasQuotient}
	The regular M\"obius transformations of $\B$ obtained from $\Spe(1,1)$ yield the map
	\begin{align} \label{eq:BasRegQuotientProposition}
		(\Spe(1) \times \{1\}) \backslash \Spe(1,1) / \Spe(1) I_2 
			&\xlongrightarrow{\sim} \B \\
		(\Spe(1) \times \{1\}) A \Spe(1) I_2 
			&\longmapsto \big(\cF_{A^{-1}}\big)^{-1}(0), \notag
	\end{align}
	which is a well-defined diffeomorphism of smooth manifolds.
\end{theorem}

\begin{remark}\label{rmk:SliceRegularBasQuotient}
	We recall that the map $\Spe(1,1) \rightarrow \cM(\B)$ given by $A \mapsto \cF_A$ is not a homomorphism. In fact, there does not seem to exist a natural group structure on the set $\cM(\B)$. For this reason, the expression $\big(\cF_{A^{-1}}\big)^{-1}$ above cannot be simplified so that the inverses cancel each other. However, as noted in subsection~\ref{subsec:RiemSymmB}, the map $\Spe(1,1) \rightarrow \Mbb(\B)$ given by $A \mapsto F_A$ is an anti-homomorphism of Lie groups. In particular, we have
	\[
		F_A = \big(F_{A^{-1}}\big)^{-1},
	\]
	for every $A \in \Spe(1,1)$. We conclude that \eqref{eq:BasSymmQuotientProposition} and 
\eqref{eq:BasRegQuotientProposition}, from Theorems~\ref{thm:RiemSymmBCentralizers} and \ref{thm:SliceRegularBasQuotient}, respectively, correspond to each other in their description of $\B$; the former considers classical M\"obius transformations, while the latter considers regular M\"obius transformations.
\end{remark}	
\begin{remark}\label{rmk:SliceQuotientBundles}
	An important feature of the diffeomorphism \eqref{eq:BasRegQuotientProposition} stated in Theorem~\ref{thm:SliceRegularBasQuotient} is that it can be obtained by taking two successive subgroup quotients on a single side. More precisely, it is proved in \cite{QB-cM(B)-asmanifold} that this double coset quotient is obtained from two smooth maps
	\[
		\Spe(1,1) \longrightarrow (\Spe(1) \times \{1\}) \backslash \Spe(1,1)
			\longrightarrow (\Spe(1) \times \{1\}) \backslash \Spe(1,1) / \Spe(1) I_2,
	\]
	where both arrows are submersions (surjective with surjective differential at every point). Furthermore, both maps are the projections of principal fiber bundles; we refer to \cite{QB-cM(B)-asmanifold} for further details and definitions. The structure group for the first arrow is $\Spe(1) \times \{1\}$ acting on the left, while the structure group for the second arrow is $\Spe(1) I_2$ acting on the right.
	
	In particular, the diffeomorphism considered in Theorem~\ref{thm:SliceRegularBasQuotient} carries an additional structure: that of principal fiber bundles with smooth group actions. It is thus natural to consider the symmetries of $\Spe(1,1)$ that are proper of this structure. The automorphisms of a principal fiber bundle are those diffeomorphisms that commute with the action of the corresponding structure group. For our setup, this leads us to consider the diffeomorphisms of $\B$ induced by left or right translations of $\Spe(1,1)$ (symmetries of the group) that commute with both the left $\Spe(1) \times \{1\}$-action and the right $\Spe(1)I_2$-action on $\Spe(1,1)$ (symmetries of the principal fiber bundles). In other words, diffeomorphisms of $\B$ from the following list
	\begin{enumerate}
		\item for $A$ lying in the centralizer of $\Spe(1) \times \{1\}$ in $\Spe(1,1)$
			\[
				(\Spe(1) \times \{1\})\; X\; \Spe(1) I_2 \longmapsto 
					(\Spe(1) \times \{1\})\; AX\; \Spe(1) I_2,
			\]	
		\item for $A$ lying in the centralizer of $\Spe(1) I_2$ in $\Spe(1,1)$
			\[
				(\Spe(1) \times \{1\})\; X \;\Spe(1) I_2 \longmapsto 
					(\Spe(1) \times \{1\})\; XA\; \Spe(1) I_2,
		\]
	\end{enumerate}
	where we identify $\B \simeq (\Spe(1) \times \{1\}) \backslash \Spe(1,1) /\Spe(1) I_2$.
\end{remark}

In order to compute the symmetries of $\B$ coming from the diffeomorphism \eqref{eq:BasRegQuotientProposition} in Theorem~\ref{thm:SliceRegularBasQuotient}, while following the reasoning of Remark~\ref{rmk:SliceQuotientBundles}, we first compute the centralizers mentioned in the latter.

\begin{lemma}\label{lem:CentralizerSp(1)x1}
	The centralizer of $\Spe(1) \times \{1\}$ in $\Spe(1,1)$ is given by
	\[
		C_{\Spe(1,1)}\big(\Spe(1) \times \{1\}\big) = \{\pm 1\} \times \Spe(1)
			= \{ \diag(\epsilon, u) : u \in \Spe(1),\; \epsilon = \pm 1 \}.
	\]
\end{lemma}
\begin{proof}
	Let $A \in \Spe(1,1)$ be the matrix
	\[
		A = \begin{pmatrix}
			a & c \\
			b & d
		\end{pmatrix}.
	\]
	Then, $A \diag(u,1) = \diag(u,1) A$ for $u \in \Spe(1)$ if and only if we have 
	\[
		au = ua,\quad bu = b,\quad uc = c.
	\]
	Hence, $A \in C_{\Spe(1,1)}\big(\Spe(1) \times \{1\}\big)$ is equivalent to $a \in \R \cap \Spe(1) = \{\pm 1\}$ and $b = c = 0$. In this case, we necessarily have $d \in \Spe(1)$ and so the result follows.
\end{proof}

\begin{lemma}\label{lem:CentralizerSp(1)I2}
	The centralizer of $\Spe(1) I_2$ in $\Spe(1,1)$ is given by
	\[
		C_{\Spe(1,1)}\big(\Spe(1) I_2\big) = \OO(1,1)  
			= \{A \in M_2(\R) : A^\top I_{1,1} A  = I_{1,1} \}.
	\]
\end{lemma}
\begin{proof}
	For a given $A \in \Spe(1,1)$, we have $A \in C_{\Spe(1,1)}\big(\Spe(1) I_2\big)$ if and only if all entries of $A$ belong to the center of $\HH$, i.e.~$A$ is a real matrix. This yields the result since $\OO(1,1) = \Spe(1,1) \cap M_2(\R)$ by the definition of $\OO(1,1)$.
\end{proof}

We now obtain the diffeomorphisms of $\B$ that come from the natural symmetries of $\Spe(1,1)$ associated to left and right translations.

\begin{proposition}\label{prop:SliceRiemBCentralizers}
	The diffeomorphisms of $\B$ obtained from 
	\[
		(\Spe(1) \times \{1\}) \backslash \Spe(1,1) / \Spe(1) I_2 \xlongrightarrow{\sim} \B,
	\]
	and a left or a right translation of $\Spe(1,1)$ that commutes with the left $\Spe(1) \times \{1\}$-action and the right $\Spe(1) I_2$-action (and so, that descend to $\B$) are given by the following actions on $\B$.
	\begin{enumerate}
		\item The left $\{\pm 1\} \times \Spe(1)$-action given by $\diag(\epsilon,u) \cdot q = \epsilon u q \overline{u}$.
		\item The right $\SO_0(1,1)$-action by (regular and) classical M\"obius transformations.
	\end{enumerate}
\end{proposition}
\begin{proof}
	We observe that a left translation $L_A$ of $\Spe(1,1)$ induces a diffeomorphism of $\B$, with the conditions considered in Remark~\ref{rmk:SliceRegularBasQuotient}, precisely when we have $A \in C_{\Spe(1,1)}(\Spe(1) \times \{1\})$. Similarly, a right translation $R_A$ of $\Spe(1,1)$ induces a diffeomorphism of $\B$, again with the restrictions stated in the referenced remark, when $A \in C_{\Spe(1,1)}(\Spe(1) I_2)$. Hence, we will compute the corresponding maps on $\B$ for these two cases using Lemmas~\ref{lem:CentralizerSp(1)x1} and \ref{lem:CentralizerSp(1)I2}, respectively.
	
	Let us assume that $A \in C_{\Spe(1,1)}(\Spe(1) \times \{1\}) = \{\pm 1\} \times \Spe(1)$. In other words, we have $A = \diag(\epsilon, u)$, where $\epsilon \in \{\pm 1\}$ and $u \in \Spe(1)$. For every $a \in \B$, let us consider the matrix
	\[
		M(a) = \frac{1}{\sqrt{1 - |a|^2}}
			\begin{pmatrix*}[r]
				1 & -\overline{a} \\
				-a & 1
			\end{pmatrix*}.
	\]
	One can easily verify that $M(a) \in \Spe(1,1)$ and that it has inverse $M(a)^{-1} = M(-a)$. It is also straightforward to compute
	\[
		\big(\cF_{M(-a)^{-1}}\big)^{-1}(0) = a.
	\]
	Hence, the inverse of the diffeomorphism considered in Theorem~\ref{thm:SliceRegularBasQuotient} is the map from $\B$ onto our double coset quotient given by $a \mapsto [M(-a)] = [M(a)^{-1}]$. Note that, in the rest of this proof, we will abbreviate the classes of the double coset quotient with brackets: $[\;\cdot\;]$. It follows that the left translation $L_{\diag(\epsilon, u)} = L_A$ induces the diffeomorphism of $\B$ obtained from the following composition of assignments
	\[
	\xymatrix{
		a \ar@{|-{>}}[r] & [M(a)^{-1}] \ar@{|-{>}}[r]^-{L_A}
		& [\diag(\epsilon, u) M(a)^{-1}] \ar@{|-{>}}[r]
		& \Big(\cF_{\big(\diag(\epsilon, u)M(a)^{-1}\big)^{-1}}\Big)^{-1}(0).
	}
	\]
	Hence, we need to compute that last expression. We observe that
	\[
		\big(\diag(\epsilon,u) M(a)^{-1}\big)^{-1} =
				M(a) \diag(\epsilon, \overline{u}) 
			=\frac{1}{\sqrt{1 - |a|^2}}
			\begin{pmatrix*}[r]
				\epsilon & -\overline{ua} \\
				-\epsilon a & \overline{u}
			\end{pmatrix*},
	\]
	and it follows from the definition of regular M\"obius transformation that
	\begin{align*}
		\cF_{\big(\diag(\epsilon, u)M(a)^{-1}\big)^{-1}}&(q) 
			= \big((\overline{u} - q \overline{ua})*(u - qua)\big)^{-1}
				\cdot (u - qua)*(q - a)  \epsilon \\
			&= \big(1 - 2 \re(a) q + q^2 |a|^2\big)^{-1} \cdot
				\big(q(u + u a^2) - (q^2 + 1) ua \big) \epsilon,
	\end{align*}
	which vanishes precisely at $ua\overline{u}$. We conclude that $L_{\diag(\epsilon,u)}$, as chosen above, induces the diffeomorphism of $\B$ given by $a \mapsto u a \overline{u}$. This yields the action from (1) in the statement restricted to the subgroup $\{1\} \times \Spe(1)$.
	
	Let us now consider the case where $A \in C_{\Spe(1,1)}(\Spe(1) I_2) = \OO(1,1)$. Following the description of $\OO(1,1)$ at the end of subsection~\ref{subsec:sliceregular}, we consider first $A = H(t) I(\epsilon)$, where 
	\[
		H(t) =
		\begin{pmatrix}
			\cosh(t) & \sinh(t) \\
			\sinh(t) & \cosh(t)
		\end{pmatrix},
		\quad
		I(\epsilon) =
		\begin{pmatrix}
			1 & 0 \\
			0 & \epsilon
		\end{pmatrix}
	\]
	for some $t \in \R$ and $\epsilon \in \{\pm 1\}$. As before, the diffeomorphism of $\B$ induced by $R_A$ is given by the following sequence of assignments
	\[
	\xymatrix{
		a \ar@{|-{>}}[r] & [M(a)^{-1}] \ar@{|-{>}}[r]^-{R_A}
			& [M(a)^{-1} H(t) I(\epsilon)] \ar@{|-{>}}[r]
			& \Big(\cF_{\big(M(a)^{-1} H(t) I(\epsilon)\big)^{-1}}\Big)^{-1}(0).
	}
	\]
	Since we have
	\begin{align*}
		\big(M(a)^{-1} H(t)& I(\epsilon)\big)^{-1} 
			= I(\epsilon) H(-t) M(a) \\
			&= \frac{1}{\sqrt{1 - |a|^2}}
				\begin{pmatrix*}[r]
					\cosh(t) & -\sinh(t) \\
					-\epsilon\sinh(t) & \epsilon\cosh(t)
				\end{pmatrix*}
				\begin{pmatrix*}[r]
					1 & -\overline{a} \\
					-a & 1
				\end{pmatrix*}  \\
			&= \frac{1}{\sqrt{1 - |a|^2}}
				\begin{pmatrix*}[r]
					\cosh(t)+\sinh(t)a & -\sinh(t)-\cosh(t)\overline{a}  \\
					-\epsilon\sinh(t)-\epsilon\cosh(t)a & 	
						\epsilon\cosh(t)+\epsilon\sinh(t)\overline{a}
				\end{pmatrix*}  \\
			&= \frac{\cosh(t)}{\sqrt{1 - |a|^2}}
				\begin{pmatrix*}[r]
					1+\tanh(t)a & -\tanh(t)-\overline{a}  \\
					-\epsilon(\tanh(t)+a) & \epsilon(1+\tanh(t)\overline{a})
				\end{pmatrix*}
	\end{align*}
	we now compute 
	\begin{align*}
		\cF_{\big(M(a)^{-1} H(t) I(\epsilon)\big)^{-1}}(q) =&  \\
			= \Big(
					\big(\epsilon(1 + \tanh(t) \overline{a})
						 -&q(\tanh(t) + \overline{a}) \big)^s 
				\Big)^{-1} \cdot  \\
			\cdot
				\Big( 
					\epsilon(1 + &\tanh(t) a)
					-q(\tanh(t) + a)				
				\Big)  \\
			* 
				\Big(
					&q(1 + \tanh(t)a)
						-\epsilon(\tanh(t) + a)
				\Big)  \\
			= \Big(
					\big(\epsilon(1 + \tanh(t) \overline{a})
					-&q(\tanh(t) + \overline{a}) \big)^s 
					\Big)^{-1} \cdot  \\
			\cdot
				\Big(
					-&\; q^2(\tanh(t) + a)(1 + \tanh(t)a) \\
					&+ q\epsilon(1+\tanh(t)a)^2 \\
					&+ q\epsilon(\tanh(t)+a)^2 \\
					&- (1+\tanh(t)a)(\tanh(t)+a)
				\Big).				
	\end{align*}
	It is straightforward to verify that the last expression vanishes at the point
	\[
		\epsilon(1 + \tanh(t)a)^{-1}(\tanh(t) + a),
	\]
	and so we conclude that, in the current case, we obtain the diffeomorphisms of $\B$ given by
	\[
		a \longmapsto \epsilon(1 + \tanh(t) a)^{-1}(a + \tanh(t)),
	\]
	where $t \in \R$ and $\epsilon \in \{\pm 1\}$. With $t = 0$ and $\epsilon = -1$, we obtain the diffeomorphism $a \mapsto -a$, thus completing the maps corresponding to the left $\{\pm 1\} \times \Spe(1)$-action stated in (1). On the other hand, if we choose $\epsilon = 1$, then we obtain the right $\SO_0(1,1)$-action stated in (2).	
	
	Finally, with our current notation and by the description of $\OO(1,1)$ at the end of subsection~\ref{subsec:sliceregular}, it remains to consider the matrix $-H(t) I(\epsilon)$. However, we have $\cF_{-H(t) I(\epsilon)} = \cF_{H(t) I(\epsilon)}$, and so we obtain diffeomorphisms of $\B$ already listed. This completes the proof.
\end{proof}

It turns out that the isometries of $(\B,g)$, obtained from suitable left and right translations of $\Spe(1,1)$, recover most of the isometry group of this Riemannian manifold. We recall our current notation
\[
	H(t) =
		\begin{pmatrix}
			\cosh(t) & \sinh(t) \\
			\sinh(t) & \cosh(t)
		\end{pmatrix},
\]
for every $t \in \R$.

\begin{theorem}\label{thm:Iso(B,g)fromSp(1,1)}
	The subgroup $I^+(\B,g)$ of orientation preserving isometries of the Riemannian manifold $(\B,g)$ is realized by the $\Spe(1) \times (\Z_2 \ltimes \SO_0(1,1))$-action on $\B$ given by
	\[
		(u,\epsilon,H(t))\cdot q 
			= \epsilon u (1 + \tanh(t)q)^{-1} (q + \tanh(t))\overline{u},
	\]
	where $u \in \Spe(1)$, $\epsilon \in \Z_2$ and $t \in \R$. Furthermore, this action corresponds to all diffeomorphisms of $\B$ that come from either left or right translations in $\Spe(1,1)$ (under the identification \eqref{eq:BasRegQuotientProposition}) that commute with the left $\Spe(1) \times \{1\}$-action and the right $\Spe(1)I_2$-action on $\Spe(1,1)$.
\end{theorem}
\begin{proof}
	Proposition~\ref{prop:SliceRiemBCentralizers} proves that the expression considered in the statement yields an action whose transformations satisfy the last claim above. This action is isometric by Theorem~\ref{thm:Iso(B,g)}.
	
	Since the groups $\Spe(1)$ and $\SO_0(1,1)$ are connected their actions are necessarily orientation preserving. The map $q \mapsto -q$ is clearly orientation preserving as well. Hence, the isometries coming from the action in the statement form a subgroup of $I^+(\B,g)$. The conjugation $q \mapsto \overline{q}$ is orientation reversing, and so Theorem~\ref{thm:Iso(B,g)} implies that the action from the statement realizes the whole group $I^+(\B,g)$.
\end{proof}

\begin{remark}\label{rmk:Iso(B,g)fromSp(1,1)}
	It follows from the observations made in Remark~\ref{rmk:Iso(B,g)} and Theorem~\ref{thm:Iso(B,g)fromSp(1,1)} that 
	\[
		I^+(\B,g) \simeq \SO(3) \times (\Z_2 \ltimes \R),
	\]
	where the semi-direct product structure on the second factor is the one described in Remark~\ref{rmk:Iso(B,g)}.
\end{remark}

We next show that the double coset representation of $\B$ given in Theorem~\ref{thm:SliceRegularBasQuotient} induces a decomposition of $\Spe(1,1)$ which is thus associated to the regular M\"obius transformations. We recall that $\frakm_\B$ was defined in subsection~\ref{subsec:RiemSymmB}.

\begin{theorem}\label{thm:G=K1BK2-fromsliceregular}
	The map $\pi : \Spe(1) \times \frakm_\B \times \Spe(1) \rightarrow \Spe(1,1)$ given by
	\[
		(u,X,v) \longmapsto \diag(u,1) \exp(X) v,
	\]
	is a diffeomorphism.
\end{theorem}
\begin{proof}
	We note that $\pi$ is a well defined smooth map.
	
	A simple substitution in the second formula from Proposition~\ref{prop:expSp(1,1)} shows that for every $q \in \HH$ we have
	\[
		\exp
			\bigg(
				\begin{pmatrix}
					0 & \overline{q} \\
					q & 0
				\end{pmatrix}
			\bigg) =
				\frac{1}{\sqrt{1 - \tanh(|q|)^2}}
					\begin{pmatrix}
						1 & \tanh(|q|) \overline{\sgn(q)} \\
						\tanh(|q|) \sgn(q) & 1
					\end{pmatrix},
	\]
	where $\sgn(q) = q/|q|$ for $q \not= 0$ and $\sgn(0) = 0$. In particular, with the notation from the proof of Proposition~\ref{prop:SliceRiemBCentralizers}, it is clear that $\exp(\frakm_\B)$ is the set of matrices $M(a)$ with $a \in \B$. The previous formula also shows that $\exp$ restricted to $\frakm_\B$ is injective and it satisfies $\exp(X)^{-1} = \exp(-X)$, for every $X \in \frakm_\B$.
	
	Recall that Theorem~\ref{thm:SliceRegularBasQuotient} yields the diffeomorphism from the double coset space $(\Spe(1) \times \{1\}) \backslash \Spe(1,1) / \Spe(1) I_2$ onto $\B$ given by
	\[
		[A] \mapsto \big(\cF_{A^{-1}}\big)^{-1}(0),
	\]
	where $[A]$ will denote in this proof the double coset class of the matrix $A$. As noted in the proof of Proposition~\ref{prop:SliceRiemBCentralizers}, and following its notation, the inverse of this diffeomorphism is given by $a \mapsto [M(-a)]$. It follows from these observations that for every $A \in \Spe(1,1)$ there exist $u,v \in \Spe(1)$ and $a \in \B$ such that
	\[
		A = \diag(u,1) M(a) v,
	\]
	thus showing that $\pi$ is surjective.

	To prove the injectivity of $\pi$, let us consider $A \in \Spe(1,1)$ for which there exist expressions
	\[
		A = \diag(u_1,1) \exp(X_1) v_1 = \diag(u_2,1) \exp(X_2) v_2,
	\]
	for some $u_1, v_1, u_2, v_2 \in \Spe(1)$ and $X_1, X_2 \in \frakm_\B$. By the remarks above and the properties of regular M\"obius transformations we conclude that
	\[
		\big(\cF_{\exp(-X_1) \diag(\overline{u}_1,1)}\big)^{-1}(0)
			= \big(\cF_{A^{-1}}\big)^{-1}(0)
			= \big(\cF_{\exp(-X_2) \diag(\overline{u}_2,1)}\big)^{-1}(0).
	\]
	On other hand, if we have
	\[
		X_j = 
			\begin{pmatrix}
				0 & \overline{q}_j \\
				q_j & 0
			\end{pmatrix}
	\]
	where $q_j \in \HH$ ($j = 1, 2$), then we obtain
	\begin{multline*}
		\exp(-X_j) \diag(\overline{u}_j,1) = \\
			= \frac{1}{\sqrt{1 - \tanh(|q_j|)^2}}
				\begin{pmatrix}
					\overline{u}_j & -\tanh(|q_j|) \overline{\sgn(q_j)} \\
					-\tanh(|q_j|) \sgn(q_j) \overline{u}_j & 1
				\end{pmatrix},	
	\end{multline*}
	for $j =1, 2$. This matrix determines a regular M\"obius transformation that vanishes precisely at $\tanh(|q_j|) \sgn(q_j)$. Hence, the expression above for $\big(\cF_{A^{-1}}\big)^{-1}(0)$ implies that $q_1 = q_2$ and so that $X_1 = X_2$. Once this has been established we arrive at 
	\[
		\diag(u_1,1) \exp(X_1) v_1 = \diag(u_2,1) \exp(X_1) v_2,
	\]
	from which an expansion of the products yields $u_1 = u_2$ and $v_1 = v_2$. This proves the injectivity of $\pi$.
	
	To finish the proof, we will compute $\pi^{-1}$ and verify its smoothness. Let us consider the map $f : \frakm_\B \rightarrow \B$ given by
	\[
		f(X) = \big(\cF_{\exp(X)^{-1}}\big)^{-1}(0).
	\]
	In other words, $f$ is the composition of the assignments
	\[
		\frakm_\B \xrightarrow{\exp} \Spe(1,1) 
			\rightarrow (\Spe(1) \times \{1\}) \backslash \Spe(1,1) / \Spe(1) I_2
			\rightarrow \B,
	\]
	where the second arrow is the natural smooth double coset quotient map and the third arrow is the diffeomorphism from Theorem~\ref{thm:SliceRegularBasQuotient}. In particular, $f$ is smooth and we will show that it is a diffeomorphism. The arguments used above show that $f$ is injective. Since $\exp(\frakm_\B)$ consists of the matrices $M(a)$, with $a \in \B$, it is clear that $f$ is surjective. 
	
	On the other hand, $\frakm_\B$ and $\B$ both have dimension $4$, and so it is enough to prove that $f$ has injective differential at every point. We observe that the first formula in this proof implies that for every $t \in \R$ and $u \in \Spe(1)$ we have
	\[
		f\bigg( t
		\begin{pmatrix}
			0 & \overline{u} \\
			u & 0
		\end{pmatrix}
		\bigg) = \tanh(t) u,
	\]
	and taking derivative with respect to $t$ this shows that
	\[
		\dif f_0\bigg(
		\begin{pmatrix}
			0 & \overline{q} \\
			q & 0
		\end{pmatrix}
		\bigg) = q,
	\]
	for every $q \in \HH$. This proves that $\dif f_0$ is injective. More generally, also from the first formula in this proof we obtain
	\[
		f(q) = \tanh(|q|) \frac{q}{|q|},
	\]
	for every $q \in \HH \setminus \{0\}$. Using polar coordinates in $\HH$, we conclude that $f$ has injective differential at every point in $\HH \setminus \{0\}$. This completes the proof that $f$ is a diffeomorphism.
	
	Let us now consider $A \in \Spe(1,1)$ and write
	\[
		A = \diag(u,1) \exp(X) v,
	\]
	where $u,v \in \Spe(1)$ and $X \in \frakm_\B$. Hence, we have
	\[
		f(X) = \big(\cF_{\exp(X)^{-1}}\big)^{-1}(0) = \big(\cF_{A^{-1}}\big)^{-1}(0)
	\]
	where the second identity is a consequence of $[A] = [\exp(X)]$. It follows that
	\[
		X = f^{-1}\Big(\big(\cF_{A^{-1}}\big)^{-1}(0)\Big) = F(A),
	\]
	which defines a smooth function $F$ of $A \in \Spe(1,1)$. As noted before, we can write $\exp(X) = M(-a)$ for some $a \in \B$. Hence, we obtain
	\[
		A = \frac{1}{\sqrt{1 - |a|^2}}
			\begin{pmatrix}
				uv & u\overline{a}v \\
				av & v
			\end{pmatrix},
	\]
	from which it follows that $u = a_{11} a_{22}^{-1}$.

	Collecting the previous computations we arrive to the following expression for the inverse of $\pi$
	\[
		\pi^{-1}(A)  
			= (a_{11} a_{22}^{-1}, F(A),
			 \exp(-F(A)) \diag(\overline{u},1) A),
	\]
	where in the last entry we use the identification $\Spe(1) I_2 \simeq \Spe(1)$. This shows the smoothness of $\pi^{-1}$ and completes the proof.
\end{proof}

\subsection{Riemannian symmetric vs slice Riemannian on $\B$}
\label{subsec:SymmVSSlice}
We close this work by collecting and comparing the information relating both classical and slice regular approaches to the geometry of $\B$ considered in this work. All the claims below follow from the results proved so far.

As we have emphasized, most of the geometry of $\B$ is obtained from the Lie group $\Spe(1,1)$. This was our starting point when we considered the M\"obius transformations, both classical and regular, which lead us to the maps
\begin{align*}
	\Spe(1,1) &\rightarrow \Mbb(\B)
		& \Spe(1,1) &\rightarrow \cM(\B) \\
	A &\mapsto F_A & A &\mapsto \cF_A.
\end{align*}
The first one yields an anti-homomorphism, while the second one is far from admitting a group-like structure. Nevertheless, both lead to quotient representations of $\B$ given by
\begin{equation}\label{eq:ComparisonQuotients}
	(\Spe(1) \times \Spe(1)) \backslash \Spe(1,1) \simeq \B \simeq
	(\Spe(1) \times \{1\}) \backslash \Spe(1,1) / \Spe(1) I_2,
\end{equation}
corresponding to the classical and regular M\"obius transformations, respectively. With this regard, it is interesting to note that
\[
	(\Spe(1) \times \{1\}) \Spe(1) I_2 = \Spe(1) \times \Spe(1)
\]
as subgroups of $\Spe(1,1)$. Hence, the same subgroup is being used to realize $\B$ as a quotient of $\Spe(1,1)$. However, such subgroup acts on one side only, in the classical case, and its action makes use of both sides, in the regular case. It is important to note that in the complex case we have
\[
	(\SU(1) \times \SU(1)) \backslash \SU(1,1) =
	(\SU(1) \times \{1\}) \backslash \SU(1,1) / \SU(1) I_2,
\]
by the commutativity of $\C$, thus ruling out the existence of a similar pair of realizations for the complex unit disk. As expected, the non-commutativity of $\HH$ plays a fundamental role.

On the other hand, the Riemannian metrics on $\B$ considered in this work are 
\[
	\widehat{g}_q(\alpha,\beta) = 
	\frac{\re(\alpha \overline{\beta})}{(1-|q|^2)^2}, \quad
	g_q(\alpha,\beta) = 
		\frac{\re((\alpha - \alpha q\alpha) 
			(\overline{\beta - \beta q \beta})}
				{|1-q^2|^2 (1-|q|^2)^2},
\]
which we have called Riemannian symmetric and slice Riemannian, respectively. The non-commutativity of $\HH$ is again the reason for these two metrics to differ. Nevertheless, both restrict to the same (canonical) complex hyperbolic metric on each complex slice.

An important feature of slice regular functions is that the real axis $\R$ plays a fundamental central role. This fact seems to be the cause for the behavior of the slice Riemannian geometry in comparison with the Riemannian symmetric one, as we now explain.

The connected component of the isometry group of $(\B,\widehat{g})$ (Riemannian symmetric geometry) is isomorphic to the group
\[
	\Spe(1,1)/\{\pm I_2\},
\]
which is also known to be the subgroup of orientation preserving isometries. On the other hand, the isometry group of $(\B,g)$ (slice Riemannian geometry) is isomorphic to the group
\[
	\SO(3) \times (\Z_2 \ltimes \SO_0(1,1)) \times \Z_2,
\]
and the orientation preserving isometries for $(\B,g)$ is isomorphic to the subgroup
\[
	\SO(3) \times (\Z_2 \ltimes \SO_0(1,1)).
\]
Here we observe a property shared by the two metrics. For both $\widehat{g}$ and $g$, the group of orientation preserving isometries is obtained from the left and right translations of $\Spe(1,1)$ that commute with the subgroup actions used to build the quotient realizations of $\B$ given by \eqref{eq:ComparisonQuotients}. 

The isometry group of $\widehat{g}$ is $10$-dimensional and acts transitively on $\B$. In contrast, the isometry group of $g$ is $4$-dimensional and the orbits are easily seen to be given by the following sets
\[
	(-1,1) = \R\cap \B, \quad
		\mathcal{O}_y 
			= \{ (1 + \tanh(t)yI)^{-1} (yI + \tanh(t)) 
					: t \in \R, \; I \in \Sbb \}.
\]
where $y \in (0,1)$. The latter are $3$-dimensional and the former is $1$-dimensional. Hence, the slice Riemannian geometry clearly distinguish the special role of the real axis.

An important feature of a Riemannian symmetric space of non-compact type is the diffeomorphism that holds between its isometry group and the product of  the space and the isotropy group. For the case of $(\B,\widehat{g})$ this yields the diffeomorphism
\begin{align*}
	\widehat{\pi} : \Spe(1) \times \Spe(1) \times \frakm_\B 
		&\longrightarrow \Spe(1,1) \\
		(u,v,X) &\longmapsto \diag(u,v) \exp(X).
\end{align*}
For the slice Riemannian geometry of $(\B,g)$ one cannot expect the same exact behavior. The reason is that the isometry group is smaller in this case. However, we have obtained the diffeomorphism
\begin{align*}
	\pi : \Spe(1) \times \frakm_\B \times \Spe(1) &\longrightarrow 		
		\Spe(1,1) \\
		(u,X,v) &\longmapsto \diag(u,1) \exp(X) v,
\end{align*}
as a consequence of the relation between $g$ and the regular M\"obius transformations. On the other hand, these diffeomorphic realizations of $\Spe(1,1)$ exhibit a strong parallelism; this can be easily verified through a close inspection of the results showing that $\widehat{\pi}$ and $\pi$ are diffeomorphisms.

\subsection*{Acknowledgment}
This research was partially supported by SNII-Secihti and by Conahcyt Grants 280732 and 61517.

\end{document}